\documentclass{amsart}\usepackage{amsmath}

\usepackage{amsfonts}

\newtheorem{theorem}{Theorem}[section]
\theoremstyle{plain}

\newtheorem{lemma}{Lemma}[section]

\numberwithin{equation}{section}
\allowdisplaybreaks

\begin{document}
\title[Neo Balcobalancing Numbers]{NEO BALCOBALANCING NUMBERS}
\author{Ahmet Tekcan}
\address{Bursa Uludag University, Faculty of Science, Department of
Mathematics, Bursa, Turkiye}
\email{tekcan@uludag.edu.tr}
\date{14.04.2025}
\subjclass[2000]{11B37, 11B39, 11D09, 11D79.}
\keywords{Balancing numbers, neo balancing numbers, neo cobalancing numbers,
triangular numbers, Pythagorean triples, Cassini identities, set of
representations, Pell equations.}

\begin{abstract}
In this work, we defined neo balcobalancing numbers, neo
Lucas-balcobalancing numbers, neo balcobalancers and neo
Lucas-balcobalancers and derived the general terms of these numbers in terms
of balancing numbers. Conversely we de\-du\-ced the general terms of
balancing, cobalancing, Lucas-ba\-lan\-cing and Lucas-cobalancing numbers in
terms of these numbers. We also deduced some relations on Binet formulas,
recurrence relations, relationship with Pell, Pell-Lucas, triangular, square
triangular numbers, Pythagorean triples and Cassini identities. We also
formulate the sum of first $n$-terms of these numbers and obtained some
formulas for the sums of Pell, Pell-Lucas, balancing and Lucas-cobalancing
numbers in terms of these numbers.
\end{abstract}

\maketitle

\section{Introduction}

A positive integer $n$ is called a balancing number (\cite{B-Pa}) if the
Diophantine equation 
\begin{equation}
1+2+\cdots +(n-1)=(n+1)+(n+2)+\cdots +(n+r)  \label{balans}
\end{equation}%
holds for some positive integer $r$ which is called balancer corresponding
to $n$. If $n$ is a balancing number with balancer $r$, then from (\ref%
{balans}) 
\begin{equation}
r=\frac{-2n-1+\sqrt{8n^{2}+1}}{2}.  \label{muty}
\end{equation}%
Though the definition of balancing numbers suggests that no balancing number
should be less than $2$. But from (\ref{muty}), they noted that $%
8(0)^{2}+1=1 $ and $8(1)^{2}+1=3^{2}$ are perfect squares. So they accepted $%
0$ and $1$ to be balancing numbers.

Later Panda and Ray (\cite{paray2}) defined that a positive integer $n$ is
called a co\-ba\-lan\-cing number if the Diophantine equation 
\begin{equation}
1+2+\cdots +n=(n+1)+(n+2)+\cdots +(n+r)  \label{cobalans}
\end{equation}
holds for some positive integer $r$ which is called cobalancer corresponding
to $n$. If $n$ is a cobalancing number with cobalancer $r$, then from (\ref%
{cobalans}) 
\begin{equation}
r=\frac{-2n-1+\sqrt{ 8n^{2}+8n+1}}{2}.  \label{muty1}
\end{equation}
From (\ref{muty1}), they noted that $8(0)^{2}+8(0)+1=1$ is a perfect square.
So they accepted $0$ to be a cobalancing number, just like Behera and Panda
accepted $0$ and $1$ to be balancing numbers.

Let $B_{n}$ denote the $n^{\text{th}}$ balancing number and let $b_{n}$
denote the $n^{\text{th}}$ cobalancing number. Then from (\ref{balans}) and (%
\ref{cobalans}), every balancing number is a cobalancer and every
cobalancing number is a balancer, that is, $B_{n}=r_{n+1}$ and $R_{n}=b_{n},$
where $R_{n}$ is the $n^{ \text{th}}$ the balancer and $r_{n}$ is the $n^{%
\text{th}}$ cobalancer. Since $R_{n}=b_{n}$, we get 
\begin{equation}
b_{n}=\frac{-2B_{n}-1+\sqrt{8B_{n}^{2}+1}}{2}\text{ \ and\ \ }B_{n}=\frac{
b_{n}+1+\sqrt{8b_{n}^{2}+8b_{n}+1}}{2}.  \label{Bb}
\end{equation}
Thus from (\ref{Bb}), $B_{n}$ is a balancing number if and only if $%
8B_{n}^{2}+1$ is a perfect square and $b_{n}$ is a cobalancing number if and
only if $8b_{n}^{2}+8b_{n}+1$ is a perfect square. Hence $C_{n}=\sqrt{
8B_{n}^{2}+1}$ and $c_{n}=\sqrt{8b_{n}^{2}+8b_{n}+1}$ are integers which are
called the Lucas-balancing number and Lucas-cobalancing number, respectively.

Let $\alpha =1+\sqrt{2}$ and $\beta =1-\sqrt{2}$ be the roots of the
characteristic equation for Pell and Pell-Lucas numbers which are the
numbers defined by $P_{0}=0,P_{1}=1,$ $P_{n}=2P_{n-1}+P_{n-2}$ and $%
Q_{0}=Q_{1}=2,Q_{n}=2Q_{n-1}+Q_{n-2}$ for $n\geq 2$. Ray proved in his PhD
thesis (\cite{raytez}) that the Binet formulas for all balancing numbers are 
$B_{n}=\frac{\alpha ^{2n}-\beta ^{2n}}{4\sqrt{2}},b_{n}=\frac{\alpha
^{2n-1}-\beta ^{2n-1}}{4\sqrt{2}}-\frac{1}{2},C_{n}=\frac{\alpha ^{2n}+\beta
^{2n}}{2}$ and $c_{n}=\frac{\alpha ^{2n-1}+\beta ^{2n-1}}{2}$ (see also \cite%
{tkcn1,olajas1,pa-ray}).

Balancing numbers and their generalizations have been investigated by
se\-ve\-ral authors from many aspects. In \cite{lip}, Liptai proved that
there is no Fi\-bo\-nac\-ci ba\-lan\-cing number except $1$ and in \cite%
{lip1} he proved that there is no Lucas balancing number. In \cite{szalay},
Szalay considered the same problem and obtained some nice results by a
different method. In \cite{tunde}, Kov\'{a}cs, Liptai and Olajos
ex\-ten\-ded the concept of balancing numbers to the $(a,b)$-balancing
numbers defined as follows: Let $a>0$ and $b\geq 0$ be coprime integers. If 
\begin{equation*}
(a+b)+\cdots +(a(n-1)+b)=(a(n+1)+b)+\cdots +(a(n+r)+b)
\end{equation*}%
for some positive integers $n$ and $r$, then $an+b$ is an $(a,b)$-balancing
number. The sequence of $(a,b)$-balancing numbers is denoted by $%
B_{m}^{(a,b)}$ for $m\geq 1$. In \cite{akos}, Liptai, Luca, Pint\'{e}r and
Szalay generalized the notion of balancing numbers to numbers defined as
follows: Let $y,k,l\in \mathbb{Z}^{+}$ such that $y\geq 4$. Then a
po\-si\-ti\-ve integer $x$ with $x\leq y-2$ is called a $(k,l)$-power
numerical center for $y$ if 
\begin{equation*}
1^{k}+\cdots +(x-1)^{k}=(x+1)^{l}+\cdots +(y-1)^{l}.
\end{equation*}%
They studied the number of solutions of the equation above and proved
several effective and ineffective finiteness results for $(k,l)$-power
numerical centers. For positive integers $k,x$, let 
\begin{equation*}
\Pi _{k}(x)=x(x+1)\dots (x+k-1).
\end{equation*}%
Then it was proved in \cite{tunde} that the equation $B_{m}=\Pi _{k}(x)$ for
fixed integer $k\geq 2$ has only infinitely many solutions and for $k\in
\{2,3,4\}$ all solutions were determined. In \cite{tengely}, Tengely
con\-si\-de\-red the case 
\begin{equation*}
B_{m}=x(x+1)(x+2)(x+3)(x+4)
\end{equation*}%
for $k=5$ and proved that this Diophantine equation has no solution for $%
m\geq 0$ and $x\in \mathbb{Z}$. In \cite{komats}, Panda, Ko\-mat\-su and
Davala considered the reciprocal sums of sequences involving balancing and
Lucas-balancing numbers. In \cite{raysums}, Ray considered the sums of
ba\-lan\-cing and Lucas-balancing numbers by matrix methods. In \cite{dash},
Dash, Ota, Dash considered the $t$-balancing numbers for an integer $t\geq 1$%
. They called that $n$ is a $t$ -ba\-lan\-cing number if the Diophantine
equation 
\begin{equation*}
1+2+\cdots +n-1=(n+1+t)+(n+2+t)+\cdots +(n+r+t)
\end{equation*}%
holds for some positive integer $r$ which is called $t$-balancer. A positive
integer $n$ is called a $t$-cobalancing number if the Diophantine equation 
\begin{equation*}
1+2+\cdots +n=(n+1+t)+(n+2+t)+\cdots +(n+r+t)
\end{equation*}%
holds for some positive integer $r$ which is called $t$-cobalancer. In \cite%
{alper}, Tekcan and Erdem determined the general terms of $t$%
-co\-ba\-lan\-cing and Lucas $t$-co\-ba\-lan\-cing numbers. In \cite{ayd},
Tekcan and Ayd\i n determined the general terms of $t$-balancing and Lucas $%
t $-balancing numbers. In \cite{almost}, Panda and Panda defined that a
positive integer $n$ is called an almost ba\-lan\-cing number if the
Di\-ophan\-ti\-ne equation 
\begin{equation*}
\left\vert \lbrack (n+1)+(n+2)+\cdots +(n+r)]-[1+2+\cdots +(n-1)]\right\vert
=1
\end{equation*}%
holds for some positive integer $r$ which is called the almost balancer. In \cite{alper1}, Tekcan and Erdem determined the general terms of all almost balancing numbers of first and second type in terms of balancing and Lucas-balancing numbers. In 
\cite{tkcnx}, Tekcan derived some results on almost balancing numbers,
triangular numbers and square triangular numbers and in \cite{tkcnxy}, Tekcan
considered the sums and spectral norms of all almost balancing numbers. In 
\cite{almosttez}, Panda defined that a positive integer $n$ is called an
almost cobalancing number if the Di\-op\-han\-ti\-ne equation 
\begin{equation*}
\left\vert \lbrack (n+1)+(n+2)+\cdots +(n+r)]-(1+2+\cdots +n)\right\vert =1
\end{equation*}%
holds for some positive integer $r$ which is called an almost cobalancer. In 
\cite{ozkoc}, \"{O}zko\c{c} and Tekcan defined $k$-balancing numbers and
derived some algebraic relations on them. In \cite{meryem, meryem1}, Tekcan
and Y\i ld\i z defined balcobalancing numbers. They sum both sides of (\ref%
{balans}) and (\ref{cobalans}) and get the Diophantine equation%
\begin{equation}
1+2+\cdots +(n-1)+1+2+\cdots +(n-1)+n=2[(n+1)+(n+2)+\cdots +(n+r)]
\label{bal-234}
\end{equation}%
and said that a positive integer $n$ is called a balcobalancing number if
the Diophantine equation in (\ref{bal-234}) verified for some positive
integer $r$ which is called balcobalancer. In \cite{ecem}, Tekcan and Akg%
\"{u}\c{c} defined almost neo cobalancing number. They said that a positive
integer $n$ is called an almost neo cobalancing number if the Diophantine
equation 
\begin{equation*}
\left\vert (n-1)+(n-0)+(n+1)+\cdots +(n+r) -(1+2+\cdots +n) \right\vert =1
\end{equation*}%
holds for some positive integer $r$ which is called almost neo cobalancer.
In \cite{tekneob}, Tekcan defined almost neo balancing number. He said that
a positive integer $n$ is called an almost neo balancing number if the
Diophantine equation 
\begin{equation*}
\left\vert (n-1)+(n-0)+(n+1)+\cdots +(n+r)-[1+2+\cdots +(n-1)] \right\vert =1
\end{equation*}%
holds for some positive integer $r$ which is called almost neo balancer. In 
\cite{ozdemir}, \"{O}zdemir introduced a new non-commutative number system
called hybrid numbers. The set of hybrid numbers, denoted by $\mathbb{\ K}$,
is defined by%
\begin{equation*}
\mathbb{K}=\{z=a+b\mathbf{i}+c\varepsilon+d\mathbf{h}:a,b,c,d\in \mathbb{R}
\},
\end{equation*}%
where $\mathbf{i}^{2}=-1,$ $\varepsilon^{2}=0,\mathbf{h}^{2}=1,\mathbf{ih}=-%
\mathbf{hi}=\varepsilon +\mathbf{i}$. In \cite{prod}, Br\'{o}d, Szynal-Liana
and W\l och defined balancing and Lucas-balancing hybrid numbers 
\begin{align*}
BH_{n}&=B_{n}+B_{n+1}\mathbf{i}+B_{n+2}\varepsilon+B_{n+3}\mathbf{h} \\
CH_{n}&=C_{n}+C_{n+1}\mathbf{i}+C_{n+2}\varepsilon+C_{n+3}\mathbf{h}.
\end{align*}%
In \cite{R-S}, Rubajczyk and Szynal-Liana defined cobalancing and
Lucas-co\-ba\-lan\-cing hybrid numbers 
\begin{align*}
bH_{n}&=b_{n}+b_{n+1}\mathbf{i}+b_{n+2}\varepsilon+b_{n+3}\mathbf{h} \\
cH_{n}&=c_{n}+c_{n+1}\mathbf{i}+c_{n+2}\varepsilon+c_{n+3}\mathbf{h}.
\end{align*}

\section{Neo Balcobalancing Numbers}

In \cite{chail}, Chailangka and Pakapongpun said that a positive integer $n$
is called a neo balancing number if the Diophantine equation 
\begin{equation}
1+2+\cdots +(n-1)=(n-1)+(n-0)+(n+1)+(n+2)+\cdots +(n+r)  \label{neobalans}
\end{equation}%
holds for some positive integer $r$ which is called neo balancer.

Later in \cite{meryemx1}, Tekcan and Y\i ld\i z defined neo cobalancing
numbers. They said that a positive integer $n$ is called a neo cobalancing
number if the Diophantine equation 
\begin{equation}
1+2+\cdots +n=(n-1)+(n-0)+(n+1)+(n+2)+\cdots +(n+r)  \label{neocobalans}
\end{equation}%
holds for some positive integer $r$ which is called neo cobalancer
corresponding to $n$.

As we did in \cite{meryem, meryem1}, if we sum both sides of (\ref{neobalans}%
) and (\ref{neocobalans}), then we get 
\begin{align}
&1+2+\cdots +(n-1)+1+2+\cdots +n=  \label{neobalcob} \\
&2[(n-1)+(n-0)+(n+1)+(n+2)+\cdots +(n+r)]  \notag
\end{align}
and said that a positive integer $n$ is called a neo balcobalancing number
if the Diophantine equation in (\ref{neobalcob}) verified for some positive
integer $r$ which is called neo balcobalancer.

From (\ref{neobalcob}), we get 
\begin{equation}
r=\frac{-2n-1+\sqrt{8n^{2}-12n+9}}{2}  \label{muter}
\end{equation}%
and 
\begin{equation}
n=2+r+\sqrt{2r^{2}+5r+2}.  \label{muter1}
\end{equation}

Let $B_{n}^{neobc}$ denote the $n^{\text{th}}$ neo balcobalancing number and
let $R_{n}^{neobc}$ denote the $n^{\text{th}}$ neo balcobalancer. Then from (%
\ref{muter}), $B_{n}^{neobc}$ is a neo balcobalancing number if and only if $%
8(B_{n}^{neobc})^{2}-12B_{n}^{neobc}+9$ is a perfect square. Thus 
\begin{equation}
C_{n}^{neobc}=\sqrt{8(B_{n}^{neobc})^{2}-12B_{n}^{neobc}+9}  \label{luk}
\end{equation}%
is an integer which is called the $n^{\text{th}}$ neo Lucas-balcobalancing
number and from (\ref{muter1}), $R_{n}^{neobc}$ is a neo balcobalancer if
and only if $2(R_{n}^{neobc})^{2}+5R_{n}^{neobc}+2$ is a perfect square.
Thus 
\begin{equation}
CR_{n}^{neobc}=\sqrt{2(R_{n}^{neobc})^{2}+5R_{n}^{neobc}+2}  \label{luk1}
\end{equation}%
is an integer which is called the $n^{\text{th}}$ neo Lucas-balcobalancer
number.

In order to determine the general terms of neo balcobalancing numbers, neo
Lucas-bal\-co\-ba\-lancing numbers, neo balcobalancers and neo
Lucas-bal\-co\-ba\-lancers, we have to determine the set of all (positive)
integer solutions of some specific Pell equation. Since $B_{n}^{neobc}$ is a neo balcobalancing number if and only if $%
8(B_{n}^{neobc})^{2}-12B_{n}^{neobc}+9$ is a perfect square, we set 
\begin{equation*}
8(B_{n}^{neobc})^{2}-12B_{n}^{neobc}+9=y^{2}
\end{equation*}%
for some positive integer $y$. Then 
\begin{equation*}
16(B_{n}^{neobc})^{2}-24B_{n}^{neobc}+18=2y^{2}
\end{equation*}%
and hence 
\begin{equation*}
(4B_{n}^{neobc}-3)^{2}+9=2y^{2}.
\end{equation*}
Taking $x=4B_{n}^{neobc}-3$, we get the Pell equation (see \cite{barb, Jacob})
\begin{equation}
x^{2}-2y^{2}=-9.  \label{mutyhuy}
\end{equation}
Let $\Omega $ denotes the set of all (positive )integer solutions of (\ref%
{mutyhuy}). Then we can give the following theorem.

\begin{theorem}
\label{teo2.1}The set of all integer solutions of $x^{2}-2y^{2}=-9$ is  
\begin{equation*}
\Omega =\{(21B_{n-1}-3B_{n-2},15B_{n-1}-3B_{n-2}):n\geq 1\}.
\end{equation*}
\end{theorem}

\begin{proof}
For the Pell equation in (\ref{mutyhuy}), the indefinite quadratic form is $%
F=(1,0,-2)$ of dis\-cri\-mi\-nant $\Delta =8$. So $\tau _{\Delta }=3+2\sqrt{2%
}$. Thus the set of representatives is $\text{Rep}=\{[\pm 3\ \ \ 3]\}$ and $%
M=\left[ 
\begin{array}{cc}
3 & 2 \\ 
4 & 3%
\end{array}
\right] $ (see \cite[p. 118]{flath}). Here we note that 
\begin{equation*}
[x_{n}\ \ \ y_{n}]=[3\ \ \ \ 3]M^{n-1}
\end{equation*}
for $n\geq 1$. Also it can be easily seen that the $n^{\text{th}}$ power of $%
M$ is 
\begin{equation*}
M^{n}=\left[ 
\begin{array}{cc}
3B_{n}-B_{n-1} & 2B_{n} \\ 
4B_{n} & 3B_{n}-B_{n-1}%
\end{array}
\right]
\end{equation*}%
for $n\geq 1$. So 
\begin{equation*}
\lbrack x_{n}\ \ \ \ y_{n}]=[3\ \ \ \ 3]M^{n-1}=[21B_{n-1}-3B_{n-2}~\ \ \ \
15B_{n-1}-3B_{n-2}].
\end{equation*}%
Thus $\Omega =\{(21B_{n-1}-3B_{n-2},15B_{n-1}-3B_{n-2}):n\geq 1\}$ as we
claimed.
\end{proof}

Before determining the general terms of neo balcobalancing numbers, neo 
Lucas-bal\-co\-ba\-lancing numbers, neo balcobalancers and neo 
Lucas-bal\-co\-ba\-lancers, we need the following Lemma first. 

\begin{lemma}\label{lem2.1}
For balancing numbers $B_{n}$, we have

\begin{enumerate}
	\item $\frac{441B_{2n-1}^{2}}{2}-63B_{2n-1}B_{2n-2}+\frac{%
		9B_{2n-2}^{2}}{2}+\frac{9}{2}=(15B_{2n-1}-3B_{2n-2})^{2}$ 
		
		\item $\frac{81B_{2n-1}^{2}}{8}-\frac{27B_{2n-1}B_{2n-2}}{4}+\frac{9B_{2n-2}^{2}}{8%
		}-\frac{9}{8}=9B_{2n-1}^{2}$
	\end{enumerate} 
for $n\geq 1$. 

\end{lemma}

\begin{proof}
Applying Binet formula $B_{n}=\frac{\alpha ^{2n}-\beta ^{2n}}{4\sqrt{2}}$,
we get the desired result. 
\end{proof}

Now from Theorem \ref{teo2.1}, we can give the following result.

\begin{theorem}
\label{teo2.2}The general terms of neo balcobalancing numbers, neo 
Lucas-bal\-co\-ba\-lancing numbers, neo balcobalancers and neo 
Lucas-bal\-co\-ba\-lancers are  
\begin{align*}
B_{n}^{neobc}& =\frac{21B_{2n-1}-3B_{2n-2}+3}{4} \\
C_{n}^{neobc}& =15B_{2n-1}-3B_{2n-2} \\
R_{n}^{neobc}& =\frac{9B_{2n-1}-3B_{2n-2}-5}{4} \\
CR_{n}^{neobc}& =3B_{2n-1}
\end{align*}
for $n\geq 1$.
\end{theorem}

\begin{proof}
Since $x=4B_{n}^{neobc}-3$, we get from Theorem 2.1 that 
\begin{equation*}
B_{n}^{neobc}=\frac{21B_{2n-1}-3B_{2n-2}+3}{4}.
\end{equation*}%
So from (\ref{luk}), we get 
\begin{align*}
C_{n}^{neobc}& =\sqrt{8(B_{n}^{neobc})^{2}-12B_{n}^{neobc}+9} \\
& =\sqrt{8(\frac{21B_{2n-1}-3B_{2n-2}+3}{4})^{2}-12(\frac{
21B_{2n-1}-3B_{2n-2}+3}{4})+9} \\
& =\sqrt{\frac{441B_{2n-1}^{2}}{2}-63B_{2n-1}B_{2n-2}+\frac{9B_{2n-2}^{2}}{2}
+\frac{9}{2}} \\
& =\sqrt{(15B_{2n-1}-3B_{2n-2})^{2}} \\
& =15B_{2n-1}-3B_{2n-2}
\end{align*}%
by (1) of Lemma \ref{lem2.1} and from (\ref{muter}) 
\begin{align*}
R_{n}^{neobc}& =\frac{-2B_{n}^{neobc}-1+C_{n}^{neobc}}{2} \\
& =\frac{-2(\frac{21B_{2n-1}-3B_{2n-2}+3}{4})-1+(15B_{2n-1}-3B_{2n-2})}{2} \\
& =\frac{9B_{2n-1}-3B_{2n-2}-5}{4}.
\end{align*}%
Finally from (\ref{luk1}), we deduce that 
\begin{align*}
CR_{n}^{neobc}& =\sqrt{2(R_{n}^{neobc})^{2}+5R_{n}^{neobc}+2} \\
& =\sqrt{2(\frac{9B_{2n-1}-3B_{2n-2}-5}{4})^{2}+5(\frac{%
9B_{2n-1}-3B_{2n-2}-5 }{4})+2} \\
& =\sqrt{\frac{81B_{2n-1}^{2}}{8}-\frac{27B_{2n-1}B_{2n-2}}{4}+\frac{
9B_{2n-2}^{2}}{8}-\frac{9}{8}} \\
& =\sqrt{9B_{2n-1}^{2}} \\
& =3B_{2n-1}
\end{align*}
by (2) of Lemma \ref{lem2.1}. This completes the proof.
\end{proof}

Here we note that neo balcobalancing numbers should be positive from
definition. But from (\ref{muter}), we see that $8(0)^{2}-12(0)+9=3^{2}$ is
a perfect square. So we accepted $0$ to be a neo balcobalancing number, that
is, $B_{0}^{neobc}=0$. In this case $R_{0}^{neobc}=1$ and $%
C_{0}^{neobc}=CR_{0}^{neobc}=3.$

\section{Binet Formulas and Recurrence Relations}

\begin{theorem}
\label{teo3.1}Binet formulas for neo balcobalancing numbers, neo 
Lucas-bal\-co\-ba\-lancing numbers, neo balcobalancers and neo 
Lucas-bal\-co\-ba\-lancers are  
\begin{align*}
B_{n}^{neobc}& =\frac{3(\alpha ^{4n-1}+\beta ^{4n-1})+6}{8} \\
C_{n}^{neobc}& =\frac{3(\alpha ^{4n-1}-\beta ^{4n-1})}{2\sqrt{2}} \\
R_{n}^{neobc}& =\frac{3(\alpha ^{4n-2}+\beta ^{4n-2})-10}{8} \\
CR_{n}^{neobc}& =\frac{3(\alpha ^{4n-2}-\beta ^{4n-2})}{4\sqrt{2}}
\end{align*}
for $n\geq 1$.
\end{theorem}

\begin{proof}
We note that $B_{n}=\frac{\alpha ^{2n}-\beta ^{2n}}{4\sqrt{2}}$. Thus from
Theorem \ref{teo2.2}, we deduce that 
\begin{align*}
B_{n}^{neobc}& =\frac{21B_{2n-1}-3B_{2n-2}+3}{4} \\
& =\frac{21(\frac{\alpha ^{4n-2}-\beta ^{4n-2}}{4\sqrt{2}})-3(\frac{\alpha
^{4n-4}-\beta ^{4n-4}}{4\sqrt{2}})+3}{4} \\
& =\frac{\frac{\alpha ^{4n}(21\alpha ^{-2}-3\alpha ^{-4})-\beta
^{4n}(21\beta ^{-2}-3\beta ^{-4})}{4\sqrt{2}}+3}{4} \\
& =\frac{\frac{\alpha ^{4n}(-3+3\sqrt{2})-\beta ^{4n}(3+3\sqrt{2})}{2}+3}{4}
\\
& =\frac{3(\alpha ^{4n-1}+\beta ^{4n-1})+6}{8}.
\end{align*}
The others can be proved similarly.
\end{proof}

\begin{theorem}
\label{teo3.2}Recurrence relations for neo balcobalancing numbers, neo 
Lucas-bal\-co\-ba\-lancing numbers, neo balcobalancers and neo 
Lucas-bal\-co\-ba\-lancers are  
\begin{align*}
B_{n}^{neobc}& =35B_{n-1}^{neobc}-35B_{n-2}^{neobc}+B_{n-3}^{neobc} \\
C_{n}^{neobc}& =35C_{n-1}^{neobc}-35C_{n-2}^{neobc}+C_{n-3}^{neobc} \\
R_{n}^{neobc}& =35R_{n-1}^{neobc}-35R_{n-2}^{neobc}+R_{n-3}^{neobc}
\end{align*}
for $n\geq 3$ and  
\begin{align*}
CR_{n}^{neobc}&=35CR_{n-1}^{neobc}-35CR_{n-2}^{neobc}+CR_{n-3}^{neobc}
\end{align*}
for $n\geq 4$.
\end{theorem}

\begin{proof}
First we notice that $35B_{2n-3}-35B_{2n-5}+B_{2n-7}=B_{2n-1}$ and $%
35B_{2n-4}-35B_{2n-6}+B_{2n-8}=B_{2n-2}$. Thus from Theorem \ref{teo2.2}, we
get 
\begin{align*}
& 35B_{n-1}^{neobc}-35B_{n-2}^{neobc}+B_{n-3}^{neobc} \\
& =35(\frac{21B_{2n-3}-3B_{2n-4}+3}{4})-35(\frac{21B_{2n-5}-3B_{2n-6}+3}{4})
\\
& \ \ \ +(\frac{21B_{2n-7}-3B_{2n-8}+3}{4}) \\
& =\frac{
21(35B_{2n-3}-35B_{2n-5}+B_{2n-7})-3(35B_{2n-4}-35B_{2n-6}+B_{2n-8})+3}{4} \\
& =\frac{21B_{2n-1}-3B_{2n-2}+3}{4} \\
& =B_{n}^{neobc}.
\end{align*}%
The others can be proved similarly.
\end{proof}

\section{From Neo Balcobalancing Numbers to Balancing Numbers}

In Theorem \ref{teo2.2}, we deduced the general terms of neo balcobalancing
numbers, neo Lucas-bal\-co\-ba\-lancing numbers, neo balcobalancers and neo
Lucas-bal\-co\-ba\-lancers in terms of balancing numbers. Conversely, we can
give the general terms of balancing, cobalancing, Lucas-balancing and
Lucas-cobalancing numbers in terms of neo balcobalancing numbers and neo
Lucas-bal\-co\-ba\-lancing numbers as follows.

\begin{theorem}
\label{teo4.1}The general terms of balancing, cobalancing, Lucas-balancing 
and Lucas-cobalancing numbers 
\begin{align*}
B_{n}& =\frac{1}{6}\left\{ 
\begin{array}{cc}
4B_{\frac{n+1}{2}}^{neobc}-C_{\frac{n+1}{2}}^{neobc}-3 & n\geq 1\text{ odd }
\\ 
&  \\ 
20B_{\frac{n+2}{2}}^{neobc}-7C_{\frac{n+2}{2}}^{neobc}-15 & n\geq 2\text{
even}%
\end{array}
\right. \\
&\\
b_{n}& =\frac{1}{12}\left\{ 
\begin{array}{cc}
-16B_{\frac{n+1}{2}}^{neobc}+6C_{\frac{n+1}{2}}^{neobc}+6 & n\geq 1\text{
odd } \\ 
&  \\ 
20B_{\frac{n+2}{2}}^{neobc}-4B_{\frac{n}{2}}^{neobc}-7C_{\frac{n+2}{2}
}^{neobc}+C_{\frac{n}{2}}^{neobc}-18 & n\geq 2\text{ even}%
\end{array}
\right. \\
&\\
C_{n}& =\frac{1}{6}\left\{ 
\begin{array}{cc}
-8B_{\frac{n+1}{2}}^{neobc}+4C_{\frac{n+1}{2}}^{neobc}+6 & n\geq 1\text{ odd 
} \\ 
&  \\ 
60B_{\frac{n+2}{2}}^{neobc}-4B_{\frac{n}{2}}^{neobc}-21C_{\frac{n+2}{2}
}^{neobc}+C_{\frac{n}{2}}^{neobc}-42 & n\geq 2\text{ even}%
\end{array}
\right. \\
&\\
c_{n}& =\frac{1}{6}\left\{ 
\begin{array}{cc}
24B_{\frac{n+1}{2}}^{neobc}-8C_{\frac{n+1}{2}}^{neobc}-18 & n\geq 1\text{
odd } \\ 
&  \\ 
20B_{\frac{n+2}{2}}^{neobc}+4B_{\frac{n}{2}}^{neobc}-7C_{\frac{n+2}{2}
}^{neobc}-C_{\frac{n}{2}}^{neobc}-18 & n\geq 2\text{ even.}%
\end{array}
\right.
\end{align*}
\end{theorem}

\begin{proof}
Let $n=2k-1$ for some positive integer $k$. Then from Th\-e\-o\-rem \ref%
{teo3.1}, we get 
\begin{align*}
B_{2k-1} &=\frac{\alpha ^{4k-2}-\beta ^{4k-2}}{4\sqrt{2}} \\
&=\frac{\frac{\alpha ^{4k-1}(3\alpha ^{-1})+\beta ^{4k-1}(-3\beta ^{-1})}{2 
\sqrt{2}}}{6} \\
&=\frac{\alpha ^{4k-1}(\frac{3}{2}-\frac{3}{2\sqrt{2}})+\beta ^{4k-1}(\frac{%
3 }{2}+\frac{3}{2\sqrt{2}})+\frac{24}{8}-3}{6} \\
&=\frac{4\left[\frac{3(\alpha ^{4k-1}+\beta ^{4k-1})+6}{8}\right]-\left[ 
\frac{3(\alpha ^{4k-1}-\beta ^{4k-1})}{2\sqrt{2}}\right]-3}{6} \\
&=\frac{4B_{k}^{neobc}-C_{k}^{neobc}-3}{6}.
\end{align*}%
Thus 
\begin{equation*}
B_{n}=\frac{4B_{\frac{n+1}{2}}^{neobc}-C_{\frac{n+1}{2}}^{neobc}-3}{6}.
\end{equation*}
The other cases can be proved similarly.
\end{proof}

\section{Relationship with Pell and Pell-Lucas Numbers}

In Theorem \ref{teo2.2}, we can give the general terms of neo balcobalancing
numbers, neo Lucas-bal\-co\-ba\-lancing numbers, neo balcobalancers and neo
Lucas-bal\-co\-ba\-lancers in terms of balancing numbers. Similarly, we can
give the general terms of these numbers in terms of Pell and also Pell-Lucas
numbers as follows.

\begin{theorem}
\label{teo5.1}The general terms of neo balcobalancing numbers, neo 
Lucas-bal\-co\-ba\-lancing numbers, neo balcobalancers and neo 
Lucas-bal\-co\-ba\-lancers are 
\begin{align*}
B_{n}^{neobc}& =\frac{21P_{4n-2}-3P_{4n-4}+6}{8} \\
C_{n}^{neobc}& =\frac{15P_{4n-2}-3P_{4n-4}}{2}
\end{align*}
\begin{align*}
R_{n}^{neobc}& =\frac{9P_{4n-2}-3P_{4n-4}-10}{8} \\
CR_{n}^{neobc}& =\frac{3P_{4n-2}}{2}
\end{align*}
for $n\geq 1$.
\end{theorem}

\begin{proof}
Since $P_{n}=\frac{\alpha ^{n}-\beta ^{n}}{2\sqrt{2}}$, we get from Theorem %
\ref{teo3.1} that 
\begin{align*}
B_{n}^{neobc}& =\frac{3(\alpha ^{4n-1}+\beta ^{4n-1})+6}{8} \\
& =\frac{3[\alpha ^{4n}(-1+\sqrt{2})+\beta ^{4n}(-1-\sqrt{2})]+6}{8} \\
& =\frac{\alpha ^{4n}(\frac{21\alpha ^{-2}-3\alpha ^{-4}}{2\sqrt{2}})+\beta
^{4n}(\frac{-21\beta ^{-2}+3\beta ^{-4}}{2\sqrt{2}})+6}{8} \\
& =\frac{21(\frac{\alpha ^{4n-2}-\beta ^{4n-2}}{2\sqrt{2}})-3(\frac{\alpha
^{4n-4}-\beta ^{4n-4}}{2\sqrt{2}})+6}{8} \\
& =\frac{21P_{4n-2}-3P_{4n-4}+6}{8}.
\end{align*}%
The other cases can be proved similarly.
\end{proof}

As in Theorem \ref{teo5.1}, we can give the following result which can be
proved similarly.

\begin{theorem}
\label{teo5.2}The general terms of neo balcobalancing numbers, neo 
Lucas-bal\-co\-ba\-lancing numbers, neo balcobalancers and neo 
Lucas-bal\-co\-ba\-lancers are 
\begin{align*}
B_{n}^{neobc}& =\frac{36Q_{4n-4}+15Q_{4n-5}+6}{8} \\
C_{n}^{neobc}& =\frac{51Q_{4n-4}+21Q_{4n-5}}{4} \\
R_{n}^{neobc}& =\frac{15Q_{4n-4}+6Q_{4n-5}-10}{8} \\
CR_{n}^{neobc}& =\frac{21Q_{4n-4}+9Q_{4n-5}}{8}
\end{align*}
for $n\geq 2$.
\end{theorem}

Conversely, we can give the even and odd ordered Pell and Pell-Lucas numbers
in terms of neo balcobalancing numbers and neo Lucas-bal\-co\-ba\-lancing
numbers as follows.

\begin{theorem}
\label{teo5.3}The general terms of the even and odd ordered Pell numbers are
\begin{align*}
P_{2n}& =\frac{1}{3}\left\{ 
\begin{array}{cc}
4B_{\frac{n+1}{2}}^{neobc}-C_{\frac{n+1}{2}}^{neobc}-3 & n\geq 1\text{ odd }
\\ 
&  \\ 
20B_{\frac{n+2}{2}}^{neobc}-7C_{\frac{n+2}{2}}^{neobc}-15 & n\geq 2\text{
even}%
\end{array}
\right. \\
& \\
P_{2n-1}& =\frac{1}{6}\left\{ 
\begin{array}{cc}
-16B_{\frac{n+1}{2}}^{neobc}+6C_{\frac{n+1}{2}}^{neobc}+12 & n\geq 1\text{
odd } \\ 
&  \\ 
20B_{\frac{n+2}{2}}^{neobc}-4B_{\frac{n}{2}}^{neobc}-7C_{\frac{n+2}{2}
}^{neobc}+C_{\frac{n}{2}}^{neobc}-12 & n\geq 2\text{ even}%
\end{array}
\right.
\end{align*}
and the general terms of the even and odd ordered Pell-Lucas numbers are  
\begin{align*}
Q_{2n}& =\frac{1}{3}\left\{ 
\begin{array}{cc}
-8B_{\frac{n+1}{2}}^{neobc}+4C_{\frac{n+1}{2}}^{neobc}+6 & n\geq 1\text{ odd 
} \\ 
&  \\ 
60B_{\frac{n+2}{2}}^{neobc}-4B_{\frac{n}{2}}^{neobc}-21C_{\frac{n+2}{2}
}^{neobc}+C_{\frac{n}{2}}^{neobc}-42 & n\geq 2\text{ even}%
\end{array}
\right. \\
& \\
Q_{2n-1}& =\frac{1}{3}\left\{ 
\begin{array}{cc}
24B_{\frac{n+1}{2}}^{neobc}-8C_{\frac{n+1}{2}}^{neobc}-18 & n\geq 1\text{
odd } \\ 
&  \\ 
20B_{\frac{n+2}{2}}^{neobc}+4B_{\frac{n}{2}}^{neobc}-7C_{\frac{n+2}{2}
}^{neobc}-C_{\frac{n}{2}}^{neobc}-18 & n\geq 2\text{ even.}%
\end{array}
\right.
\end{align*}
\end{theorem}

\begin{proof}
It can be proved in the same way that Theorem \ref{teo5.1} was proved.
\end{proof}

Thus we construct one-to-one correspondence between all neo balcobalancing
numbers and Pell and Pell-Lucas numbers.

\section{Relationship with Triangular and Square Triangular Numbers}

Recall that triangular numbers denoted by $T_{n}$ are the numbers of the
form 
\begin{equation*}
T_{n}=\frac{n(n+1)}{2}.
\end{equation*}%
It is known that there is a correspondence between balancing numbers and
triangular numbers. From (\ref{balans}), we see that $n$ is a balancing number if and only if $n^{2}$ is a triangular number since
\begin{equation*}
\frac{(n+r)(n+r+1)}{2}=n^{2}.
\end{equation*}%
So 
\begin{equation}
T_{B_{n}+R_{n}}=B_{n}^{2}.  \label{lkuyt}
\end{equation}%
As in (\ref{lkuyt}), we can give the following theorem.

\begin{theorem}
\label{teo6.1}$B_{n}^{neobc}$ is a neo balcobalancing number if and only if $
(B_{n}^{neobc})^{2}-\frac{3B_{n}^{neobc}}{2}+1$ is a triangular number, that
is,  
\begin{equation*}
T_{B_{n}^{neobc}+R_{n}^{neobc}}=(B_{n}^{neobc})^{2}-\frac{3B_{n}^{neobc}}{2}
+1.
\end{equation*}
\end{theorem}

\begin{proof}
Let $n$ be a neo balcobalancing number. Then from (\ref{neobalcob}), we get 
\begin{equation*}
\frac{(n-1)n}{2}+\frac{n(n+1)}{2}=2\left[ 2n-1+nr+\frac{r(r+1)}{2}\right]
\end{equation*}%
and hence 
\begin{equation*}
\frac{(n+r)(n+r+1)}{2}=n^{2}-\frac{3n}{2}+1.
\end{equation*}%
Thus 
\begin{equation*}
T_{B_{n}^{neobc}+R_{n}^{neobc}}=(B_{n}^{neobc})^{2}-\frac{3B_{n}^{neobc}}{2}
+1
\end{equation*}%
as we claimed.
\end{proof}

There are infinitely many triangular numbers that are also square numbers
which are called square triangular numbers and is denoted by $S_{n}$. Notice
that 
\begin{equation*}
S_{n}=s_{n}^{2}=\frac{t_{n}(t_{n}+1)}{2},
\end{equation*}%
where $s_{n}$ and $t_{n}$ are the sides of the corresponding square and
triangle. Their Binet formulas are 
\begin{equation}
S_{n}=\frac{\alpha ^{4n}+\beta ^{4n}-2}{32}\text{, }s_{n}=\frac{\alpha
^{2n}-\beta ^{2n}}{4\sqrt{2}}\text{ and }t_{n}=\frac{\alpha ^{2n}+\beta
^{2n}-2}{4}  \label{cvtyu}
\end{equation}%
for $n\geq 1$.

We can give the general terms of neo balcobalancing numbers, neo
Lucas-bal\-co\-ba\-lancing numbers, neo balcobalancers and neo
Lucas-bal\-co\-ba\-lancers in terms of square triangular numbers as follows:

\begin{theorem}
\label{teo6.2}The general terms of neo balcobalancing numbers, neo 
Lucas-bal\-co\-ba\-lancing numbers, neo balcobalancers and neo 
Lucas-bal\-co\-ba\-lancers are  
\begin{align*}
B_{n}^{neobc}& =\frac{6s_{2n-1}+3t_{2n-1}+3}{2} \\
C_{n}^{neobc}& =6s_{2n-1}+6t_{2n-1}+3 \\
R_{n}^{neobc}& =\frac{3t_{2n-1}-1}{2} \\
CR_{n}^{neobc}& =3s_{2n-1}
\end{align*}
for $n\geq 1$.
\end{theorem}

\begin{proof}
From Theorem \ref{teo3.1}, we get 
\begin{align*}
B_{n}^{neobc}& =\frac{3(\alpha ^{4n-1}+\beta ^{4n-1})+6}{8} \\
& =\frac{\frac{3}{4}(\alpha ^{4n-1}+\beta ^{4n-1})+\frac{3}{2}}{2} \\
& =\frac{\alpha ^{4n-1}[\frac{6\alpha ^{-1}}{4\sqrt{2}}+\frac{3\alpha ^{-1}}{
4}]+\beta ^{4n-1}[\frac{-6\beta ^{-1}}{4\sqrt{2}}+\frac{3\beta ^{-1}}{4}]+ 
\frac{3}{2}}{2} \\
& =\frac{6(\frac{\alpha ^{4n-2}-\beta ^{4n-2}}{4\sqrt{2}})+3(\frac{\alpha
^{4n-2}+\beta ^{4n-2}-2}{4})+3}{2} \\
& =\frac{6s_{2n-1}+3t_{2n-1}+3}{2}
\end{align*}
by (\ref{cvtyu}). The others can be proved similarly.
\end{proof}

Conversely, we can give the following theorem.

\begin{theorem}
\label{teo6.3}The general terms of $S_{n},s_{n}$ and $t_{n}$ are 
\begin{align*}
S_{n}& =\frac{1}{36}\left\{ 
\begin{array}{cc}
(4B_{\frac{n+1}{2}}^{neobc}-C_{\frac{n+1}{2}}^{neobc}-3)^{2} & n\geq 1\text{
odd } \\ 
&  \\ 
(20B_{\frac{n+2}{2}}^{neobc}-7C_{\frac{n+2}{2}}^{neobc}-15)^{2} & n\geq 2 
\text{ even}%
\end{array}
\right. \\
& \\
s_{n}& =\frac{1}{6}\left\{ 
\begin{array}{cc}
4B_{\frac{n+1}{2}}^{neobc}-C_{\frac{n+1}{2}}^{neobc}-3 & n\geq 1\text{ odd }
\\ 
&  \\ 
20B_{\frac{n+2}{2}}^{neobc}-7C_{\frac{n+2}{2}}^{neobc}-15 & n\geq 2\text{
even}%
\end{array}
\right. \\
& \\
t_{n}& =\frac{1}{12}\left\{ 
\begin{array}{cc}
-8B_{\frac{n+1}{2}}^{neobc}+4C_{\frac{n+1}{2}}^{neobc} & n\geq 1\text{ odd }
\\ 
&  \\ 
60B_{\frac{n+2}{2}}^{neobc}-4B_{\frac{n}{2}}^{neobc}-21C_{\frac{n+2}{2}
}^{neobc}+C_{\frac{n}{2}}^{neobc}-48 & n\geq 2\text{ even.}%
\end{array}
\right.
\end{align*}
\end{theorem}

\begin{proof}
Let $n=2k-1$ for some positive integer $k$. Then from (\ref{cvtyu}), we get 
\begin{align*}
S_{2k-1}& =\frac{\alpha ^{8k-4}+\beta ^{8k-4}-2}{32} \\
& =\frac{\frac{9}{8}[\alpha ^{8k-4}-2(\alpha \beta )^{4k-2}+\beta ^{8k-4}]}{
36} \\
& =\frac{\left[ \frac{3}{2\sqrt{2}}(\alpha ^{4k-2}-\beta ^{4k-2})\right]
^{2} }{36} \\
& =\frac{\left[ \alpha ^{4k-1}(\frac{12}{8}-\frac{3}{2\sqrt{2}})+\beta
^{4k-1}(\frac{12}{8}+\frac{3}{2\sqrt{2}})\right] ^{2}}{36} \\
& =\frac{\left[ 4\left( \frac{3(\alpha ^{4k-1}+\beta ^{4k-1})+6}{8}\right)
-\left( \frac{3(\alpha ^{4k-1}-\beta ^{4k-1})}{2\sqrt{2}}\right) -3\right]
^{2}}{36} \\
& =\frac{(4B_{k}^{neobc}-C_{k}^{neobc}-3)^{2}}{36}
\end{align*}%
by Theorem \ref{teo3.1}. So 
\begin{equation*}
S_{n}=\frac{(4B_{\frac{n+1}{2}}^{neobc}-C_{\frac{n+1}{2}}^{neobc}-3)^{2}}{36}
\end{equation*}
for odd $n\geq 1$. The others can be proved similarly.
\end{proof}

Thus we construct one-to-one correspondence between all neo balcobalancing
numbers and square triangular numbers.

\section{Pythagorean Triples}

Notice that a Pythagorean triple consists of three positive integers $a,b,c$
such that $a^{2}+b^{2}=c^{2}$ and commonly written $(a,b,c)$. It is known
that 
\begin{equation*}
(2P_{n}P_{n+1},P_{n+1}^{2}-P_{n}^{2},P_{n+1}^{2}+P_{n}^{2})
\end{equation*}%
is a Pythagorean triple. Similarly we can give the following result.

\begin{theorem}
\label{teo7.1}If $n\geq 2$ is even, then  
\begin{align*}
\left(\frac{12B_{\frac{n+2}{2}}^{neobc}-4C_{\frac{n+2}{2}}^{neobc}-6}{6},%
\frac{12B_{\frac{n+2}{2}}^{neobc}-4C_{\frac{n+2}{2}}^{neobc}-12}{6},\frac{%
-16B_{\frac{n+2}{2}}^{neobc}+6C_{\frac{n+2}{2}}^{neobc}+12}{6}\right)
\end{align*}
is a Pythagorean triple and if $n\geq 1$ is odd, then $(a,b,c)$ is a
Pythagorean triple, where  
\begin{align*}
a &=\frac{20B_{\frac{n+3}{2}}^{neobc}+4B_{\frac{n+1}{2}}^{neobc}-7C_{\frac{
n+3}{2}}^{neobc}-C_{\frac{n+1}{2}}^{neobc}-12}{12} \\
b &=\frac{20B_{\frac{n+3}{2}}^{neobc}+4B_{\frac{n+1}{2}}^{neobc}-7C_{\frac{
n+3}{2}}^{neobc}-C_{\frac{n+1}{2}}^{neobc}-24}{12} \\
c &=\frac{20B_{\frac{n+3}{2}}^{neobc}-4B_{\frac{n+1}{2}}^{neobc}-7C_{\frac{
n+3}{2}}^{neobc}+C_{\frac{n+1}{2}}^{neobc}-12}{6}.
\end{align*}
\end{theorem}

\begin{proof}
Let $n=2k$ for some positive integer $k$. Then from Theorem \ref{teo3.1}, we
get 
\begin{align*}
& \left( \frac{12B_{k+1}^{neobc}-4C_{k+1}^{neobc}-6}{6}\right) ^{2}+\left( 
\frac{12B_{k+1}^{neobc}-4C_{k+1}^{neobc}-12}{6}\right) ^{2} \\
& =8(B_{k+1}^{neobc})^{2}-\frac{16B_{k+1}^{neobc}C_{k+1}^{neobc}}{3}
-12B_{k+1}^{neobc}+\frac{8(C_{k+1}^{neobc})^{2}}{9}+4C_{k+1}^{neobc}+5 \\
& =8\left[ \frac{3(\alpha ^{4k+3}+\beta ^{4k+3})+6}{8}\right] ^{2}-\frac{16}{
3}\left[ \frac{3(\alpha ^{4k+3}+\beta ^{4k+3})+6}{8}\right] \left[ \frac{
3(\alpha ^{4k+3}-\beta ^{4k+3})}{2\sqrt{2}}\right] \\
&-12\left[ \frac{3(\alpha ^{4k+3}+\beta ^{4k+3})+6}{8}\right] +\frac{8 }{9}%
\left[ \frac{3(\alpha ^{4k+3}-\beta ^{4k+3})}{2\sqrt{2}}\right] ^{2}+4\left[ 
\frac{3(\alpha ^{4k+3}-\beta ^{4k+3})}{2\sqrt{2}}\right] +5 \\
& =\frac{17}{8}(\alpha ^{8k+6}+\beta ^{8k+6})-\frac{3}{\sqrt{2}}(\alpha
^{8k+6}-\beta ^{8k+6})+\frac{\alpha ^{4k+3}\beta ^{4k+3}}{4}+\frac{1}{2} \\
& =\frac{17}{8}(\alpha ^{8k+6}+\beta ^{8k+6})-\frac{3}{\sqrt{2}}(\alpha
^{8k+6}-\beta ^{8k+6})-\frac{\alpha ^{4k+3}\beta ^{4k+3}}{4} \\
& =\frac{64}{9}\left[ \frac{3(\alpha ^{4k+3}+\beta ^{4k+3})+6}{8}\right]
^{2}-\frac{16}{3}\left[ \frac{3(\alpha ^{4k+3}+\beta ^{4k+3})+6}{8}\right] %
\left[ \frac{3(\alpha ^{4k+3}-\beta ^{4k+3})}{2\sqrt{2}}\right] \\
&-\frac{32}{3}\left[ \frac{3(\alpha ^{4k+3}+\beta ^{4k+3})+6}{8} \right] +%
\left[ \frac{3(\alpha ^{4k+3}-\beta ^{4k+3})}{2\sqrt{2}}\right] ^{2} +4\left[
\frac{3(\alpha ^{4k+3}-\beta ^{4k+3})}{2\sqrt{2}}\right] +4 \\
& =\frac{64}{9}(B_{k+1}^{neobc})^{2}-\frac{16B_{k+1}^{neobc}C_{k+1}^{neobc}}{
3}-\frac{32}{3}B_{k+1}^{neobc}+(C_{k+1}^{neobc})^{2}+4C_{k+1}^{neobc}+4 \\
& =\left( \frac{-16B_{k+1}^{neobc}+6C_{k+1}^{neobc}+12}{6}\right) ^{2}.
\end{align*}%
The other case can be proved similarly.
\end{proof}

\section{Cassini Identities}

Note that Cassini identity for Fibonacci numbers $F_{n}$ is 
\begin{equation*}
F_{n-1}F_{n+1}-F_{n}^{2}=(-1)^{n}
\end{equation*}%
for $n\geq 1$. Similarly for neo balcobalancing numbers, neo
Lucas-bal\-co\-ba\-lancing numbers, neo balcobalancers and neo
Lucas-bal\-co\-ba\-lancers, we can give the following result.

\begin{theorem}
\label{teo8.1}Cassini identities for neo balcobalancing numbers, neo 
Lucas-bal\-co\-ba\-lancing numbers, neo balcobalancers and neo 
Lucas-bal\-co\-ba\-lancers are  
\begin{align*}
B_{n-1}^{neobc}B_{n+1}^{neobc}-(B_{n}^{neobc})^{2}& =\frac{
3B_{n+1}^{neobc}-6B_{n}^{neobc}+3B_{n-1}^{neobc}-648}{4} \\
C_{n-1}^{neobc}C_{n+1}^{neobc}-(C_{n}^{neobc})^{2}& =1296 \\
R_{n-1}^{neobc}R_{n+1}^{neobc}-(R_{n}^{neobc})^{2}& =\frac{
-5R_{n+1}^{neobc}+10R_{n}^{neobc}-5R_{n-1}^{neobc}+648}{4}
\end{align*}
for $n\geq 1$ and  
\begin{equation*}
CR_{n-1}^{neobc}CR_{n+1}^{neobc}-(CR_{n}^{neobc})^{2}=-324
\end{equation*}
for $n\geq 2.$
\end{theorem}

\begin{proof}
Applying Theorem \ref{teo3.1}, we deduce that 
\begin{align*}
& B_{n-1}^{neobc}B_{n+1}^{neobc}-(B_{n}^{neobc})^{2} \\
& =\left( \frac{3(\alpha ^{4n-5}+\beta ^{4n-5})+6}{8}\right) \left( \frac{
3(\alpha ^{4n+3}+\beta ^{4n+3})+6}{8}\right) -\left( \frac{3(\alpha
^{4n-1}+\beta ^{4n-1})+6}{8}\right) ^{2} \\
& =\frac{18(\alpha ^{4n+3}+\beta ^{4n+3})-36(\alpha ^{4n-1}+\beta
^{4n-1})+18(\alpha ^{4n-5}+\beta ^{4n-5})-10368}{64} \\
& =\frac{3\left( \frac{3(\alpha ^{4n+3}+\beta ^{4n+3})+6}{8}\right) -6\left( 
\frac{3(\alpha ^{4n-1}+\beta ^{4n-1})+6}{8}\right) +3\left( \frac{3(\alpha
^{4n-5}+\beta ^{4n-5})+6}{8}\right) -648}{4} \\
& =\frac{3B_{n+1}^{neobc}-6B_{n}^{neobc}+3B_{n-1}^{neobc}-648}{4}.
\end{align*}%
The others can be proved similarly.
\end{proof}

\section{Sums}

\begin{theorem}
\label{teo9.1}The sum of first $n$-terms of neo balcobalancing numbers, neo 
Lucas-bal\-co\-ba\-lancing numbers, neo balcobalancers and neo 
Lucas-bal\-co\-ba\-lancers are  
\begin{align*}
\sum_{i=1}^{n}B_{i}^{neobc}& =\frac{21B_{n}^{2}-3B_{n}B_{n+1}+3B_{2n}+3n}{4}
\\
\sum_{i=1}^{n}C_{i}^{neobc}& =15B_{n}^{2}-3B_{n}B_{n+1}+3B_{2n} \\
\sum_{i=1}^{n}R_{i}^{neobc}& =\frac{9B_{n}^{2}-3B_{n}B_{n+1}+3B_{2n}-5n}{4}
\\
\sum_{i=1}^{n}CR_{i}^{neobc}& =3B_{n}^{2}
\end{align*}
for $n\geq 1$.
\end{theorem}

\begin{proof}
From Theorem \ref{teo2.2}, we get 
\begin{align*}
\sum_{i=1}^{n}B_{i}^{neobc}& =\sum_{i=1}^{n}\left(\frac{%
21B_{2i-1}-3B_{2i-2}+3}{4} \right) \\
& =\frac{21B_{1}-3B_{0}+3}{4}+\frac{21B_{3}-3B_{2}+3}{4}+\cdots +\frac{
21B_{2n-1}-3B_{2n-2}+3}{4} \\
& =\frac{21(B_{1}+B_{3}+\cdots +B_{2n-1})-3(B_{0}+B_{2}+\cdots +B_{2n-2})+3n 
}{4} \\
& =\frac{21(B_{1}+B_{3}+\cdots +B_{2n-1})-3(B_{2}+\cdots
+B_{2n-2}+B_{2n}-B_{2n})+3n}{4} \\
& =\frac{21B_{n}^{2}-3B_{n}B_{n+1}+3B_{2n}+3n}{4}.
\end{align*}%
The others can be proved similarly.
\end{proof}

In \cite[Lemma 1]{S-D}, Santana and Diaz-Barrero proved that the sum of
first nonzero $4n+1$ terms of Pell numbers is a perfect square and is 
\begin{equation*}
\sum_{i=1}^{4n+1}P_{i}=\left[ \sum_{i=0}^{n}\left( 
\begin{array}{c}
2n+1 \\ 
2i%
\end{array}
\right) 2^{i}\right] ^{2}.
\end{equation*}
Similarly we can give the following result.

\begin{theorem}
\label{teo9.2}The sum of first nonzero $8n-3$ terms of Pell numbers is a 
perfect square and is  
\begin{equation*}
\sum_{i=1}^{8n-3}P_{i}=\left(\frac{
20B_{n+1}^{neobc}+2B_{n}^{neobc}-7C_{n+1}^{neobc}-2R_{n}^{neobc}-19}{6}
\right)^{2}
\end{equation*}
for $n\geq 1.$
\end{theorem}

\begin{proof}
From Theorem \ref{teo2.2}, we get 
\begin{equation}
B_{2n}=\frac{20B_{n+1}^{neobc}-7C_{n+1}^{neobc}-15}{6}\ \ \ \text{and} \ \ \
B_{2n-1}=\frac{B_{n}^{neobc}-R_{n}^{neobc}-2}{3}.  \label{byuter}
\end{equation}%
Since $\underset{i=1}{\overset{n}{\sum }}P_{i}=\frac{P_{n}+P_{n+1}-1}{2}$
and $P_{n}=\frac{\alpha ^{n}-\beta ^{n}}{2\sqrt{2}}$, we get 
\begin{align*}
\underset{i=1}{\overset{8n-3}{\sum }}P_{i}& =\frac{P_{8n-3}+P_{8n-2}-1}{2} \\
& =\frac{\frac{\alpha ^{8n-3}-\beta ^{8n-3}}{2\sqrt{2}}+\frac{\alpha
^{8n-2}-\beta ^{8n-2}}{2\sqrt{2}}-1}{2} \\
& =\frac{\frac{\alpha ^{8n-2}(1+\alpha ^{-1})+\beta ^{8n-2}(-1-\beta ^{-1})}{
2\sqrt{2}}}{2}-\frac{1}{2} \\
& =\frac{\alpha ^{8n-2}+\beta ^{8n-2}-2}{4} \\
& =\left(\frac{\alpha ^{4n-1}+\beta ^{4n-1}}{2}\right) ^{2} \\
& =\left(\frac{\alpha ^{4n-1}(\alpha +\alpha ^{-1})+\beta ^{4n-1}(-\beta
-\beta ^{-1})}{4\sqrt{2}}\right) ^{2} \\
& =\left( \frac{\alpha ^{4n}-\beta ^{4n}}{4\sqrt{2}}+\frac{\alpha
^{4n-2}-\beta ^{4n-2}}{4\sqrt{2}}\right) ^{2} \\
& =(B_{2n}+B_{2n-1})^{2} \\
& =\left(\frac{20B_{n+1}^{neobc}-7C_{n+1}^{neobc}-15}{6}+\frac{
B_{n}^{neobc}-R_{n}^{neobc}-2}{3}\right) ^{2} \\
& =\left(\frac{
20B_{n+1}^{neobc}+2B_{n}^{neobc}-7C_{n+1}^{neobc}-2R_{n}^{neobc}-19}{6}
\right) ^{2}
\end{align*}%
by (\ref{byuter}).
\end{proof}

As in Theorem \ref{teo9.2}, we can give the following result which can be
proved similarly.

\begin{theorem}
\label{teo9.3}For the sums of balancing, Lucas-cobalancing, Pell and 
Pell-Lucas numbers, we have  
\begin{align*}
\sum_{i=1}^{2n}B_{2i-1}& =\left( \frac{20B_{n+1}^{neobc}-7C_{n+1}^{neobc}-15 
}{6}\right) ^{2} \\
1+\sum_{i=1}^{4n+2}c_{i}& =\left( \frac{4R_{n+1}^{neobc}+5}{3}\right) ^{2} \\
1+\sum_{i=1}^{8n-5}P_{i}& =\left( \frac{-4B_{n}^{neobc}+2C_{n}^{neobc}+3}{3}
\right) ^{2} \\
\sum_{i=1}^{4n}Q_{2i-1}& =\left( \frac{%
40B_{n+1}^{neobc}-14C_{n+1}^{neobc}-30 }{3}\right) ^{2} \\
\frac{\overset{4n}{\underset{i=0}{\sum }}Q_{2i+1}}{2}& =\left( \frac{
12B_{n+1}^{neobc}-4C_{n+1}^{neobc}-9}{3}\right) ^{2}.
\end{align*}
\end{theorem}

Panda and Ray proved in \cite[Theorem 3.4]{pa-ray} that the sum of first $%
2n-1$ Pell num\-bers is equal to the sum of the $n^{\text{th}}$ balancing
number and the $n^{\text{th}}$ cobalancing number, that is, 
\begin{equation}
\underset{i=1}{\overset{2n-1}{\sum }}P_{i}=B_{n}+b_{n}.  \label{xsert}
\end{equation}%
Later G\"{o}zeri, \"{O}zko\c{c} and Tekcan proved in \cite[Theorem 2.5]%
{tkcn1} that the sum of Pell-Lucas numbers from $0$ to $2n-1$ is equal to
the the sum of $n^{\text{th}}$ Lucas-balancing and the $n^{\text{th}}$
Lucas-cobalancing number, that is, 
\begin{equation*}
\underset{i=0}{\overset{2n-1}{\sum }}Q_{i}=C_{n}+c_{n}.
\end{equation*}%
Since $b_{n}=R_{n}$, (\ref{xsert}) becomes 
\begin{equation}
\sum_{i=1}^{2n-1}P_{i}=B_{n}+R_{n}.  \label{butyert}
\end{equation}%
As in (\ref{butyert}), we can give the following result.

\begin{theorem}
\label{teo9.4}The sum of first $2n-1$ even ordered Pell num\-bers $
+\,P_{4n-1} $ is equal to the sum of $n^{\text{th}}$ neo balcobalancing
number  and its balancer, that is, 
\begin{equation*}
\sum_{i=1}^{2n-1}P_{2i}+P_{4n-1}=B_{n}^{neobc}+R_{n}^{neobc}
\end{equation*}
for $n\geq 1$.
\end{theorem}

\begin{proof}
Since $P_{n}=\frac{\alpha ^{n}-\beta ^{n}}{2\sqrt{2}},$ we deduce that 
\begin{align*}
\sum_{i=1}^{2n-1}P_{2i}+P_{4n-1}& =\sum_{i=1}^{2n-1}\left( \frac{\alpha
^{2i}-\beta ^{2i}}{2\sqrt{2}}\right) +\frac{\alpha ^{4n-1}-\beta ^{4n-1}}{2 
\sqrt{2}} \\
& =\frac{\frac{\alpha ^{4n}-\alpha ^{2}}{^{\alpha ^{2}-1}}-\frac{\beta
^{4n}-\beta ^{2}}{\beta ^{2}-1}}{2\sqrt{2}}+\frac{\alpha ^{4n-1}-\beta
^{4n-1}}{2\sqrt{2}} \\
& =\frac{\frac{\alpha (\alpha ^{4n-2}-1)}{2}-\frac{\beta (\beta ^{4n-2}-1)}{%
2 }}{2\sqrt{2}}+\frac{\alpha ^{4n-1}-\beta ^{4n-1}}{2\sqrt{2}} \\
& =\frac{3(\alpha ^{4n-1}-\beta ^{4n-1})}{4\sqrt{2}}-\frac{1}{2} \\
& =\frac{3\alpha ^{4n-1}(1+\alpha ^{-1})-3\beta ^{4n-1}(-1-\beta ^{-1})}{8}- 
\frac{1}{2} \\
& =\frac{3(\alpha ^{4n-1}+\beta ^{4n-1})+6}{8}+\frac{3(\alpha ^{4n-2}+\beta
^{4n-2})-10}{8} \\
& =B_{n}^{neobc}+R_{n}^{neobc}
\end{align*}%
by Theorem \ref{teo3.1}.
\end{proof}


\begin{thebibliography}{99}
\bibitem{barb} E.J. Barbeau. \textit{Pell's Equation.} Springer-Verlag, New 
York, 2003.

\bibitem{B-Pa} A. Behera and G.K. Panda. \textit{On the Square Roots of
Triangular Numbers}. The Fibonacci Quarterly \textbf{37}(2)(1999), 98-105.

\bibitem{prod} D. Br\'{o}d, A. Szynal-Liana and I. W\l och. \textit{%
Balancing Hybrid Numbers, their Properties and Some Identities.} Indian J. 
Math. \textbf{63}(1)(2021), 143-157.

\bibitem{chail} N. Chailangka and A. Pakapongpun. \textit{Neo Balancing
Numbers}. Int. Jour. of Mathematics and Computer Science \textbf{\ 16}%
(4)(2021), 1653-1664.

\bibitem{dash} K.K. Dash, R.S. Ota and S. Dash. \textit{$t$-Balancing Numbers%
}. Int. J. Contemp. Math. Sciences \textbf{7} (41)(2012), 1999-2012.

\bibitem{flath} D.E. Flath. \textit{Introduction to Number Theory.} Wiley, 
1989.

\bibitem{tkcn1} G.K. G\"{o}zeri, A. \"{O}zko\c{c} and A. Tekcan. \textit{%
Some Algebraic Relations on Balancing Numbers}. Utilitas Mathematica \textbf{%
\ 103}(2017), 217-236.

\bibitem{tunde} T. Kovacs, K. Liptai and P. Olajos. \textit{On $(a,b)$%
-Balancing Numbers}. Publ. Math. Deb. \textbf{77}(3-4) (2010),  485-498.

\bibitem{akos} K. Liptai, F. Luca, A. Pinter and L. Szalay. \textit{%
Generalized Balancing Numbers}. Indag. Math. \textbf{20}(1)(2009), 87-100.

\bibitem{lip} K. Liptai. \textit{Fibonacci Balancing Numbers}. The Fibonacci
Quarterly \textbf{42}(4)(2004), 330-340.

\bibitem{lip1} K. Liptai. \textit{Lucas Balancing Numbers}. Acta Math. Univ.
Ostrav. \textbf{14}(2006), 43-47.

\bibitem{Jacob} M. Jacobson and K. Williams. \textit{Solving the Pell
Equation}. CMS Books in Mathematics. Springer Science, Busines Media, LLC. 
2009.

\bibitem{olajas1} P. Olajos. \textit{Properties of Balancing, Cobalancing
and Generalized Balancing Numbers}. Annales Mathematicae et Informaticae  
\textbf{37}(2010), 125-138.

\bibitem{ozdemir} M. \"{O}zdemir. \textit{Introduction to Hybrid Numbers}. 
Adv. Appl. Clifford Algebr. \textbf{28}(1)(2018), Paper No. 11, 32 pp.

\bibitem{ozkoc} A. \"{O}zko\c{c} and A. Tekcan. \textit{On $k$-Balancing
Numbers.} Notes on Number Theory and Discrete Maths. \textbf{23}(3)(2017), 
38-52.

\bibitem{pa-ray} G.K. Panda and P.K. Ray. \textit{Some Links of Balancing
and Cobalancing Numbers with Pell and Associated Pell Numbers}. Bul. of 
Inst. of Math. Acad. Sinica \textbf{6}(1)(2011), 41-72.

\bibitem{paray2} G.K. Panda and P.K. Ray. \textit{Cobalancing Numbers and
Cobalancers}. Int. J. Math. Math. Sci. \textbf{8}(2005), 1189-1200.

\bibitem{almost} G.K. Panda and A.K. Panda. \textit{Almost Balancing Numbers}%
. Journal of the Indian Math. Soc. \textbf{82}(3-4)(2015), 147-156.

\bibitem{komats} G.K. Panda, T. Komatsu and R.K. Davala. \textit{Reciprocal
Sums of Sequences Involving Balancing and Lucas-Balancing Numbers} . Math.
Reports \textbf{20}(70)(2018), 201-214.

\bibitem{almosttez} A.K. Panda. \textit{Some Variants of the Balancing
Sequences}. Ph.D. thesis, National Institute of Technology Rourkela, India, 
2017.

\bibitem{raytez} P.K. Ray. \textit{Balancing and Cobalancing Numbers}. 
Ph.D. thesis, Department of Mathematics, National Institute of Technology, 
Rourkela, India, 2009.

\bibitem{raysums} P.K. Ray. \textit{Balancing and Lucas-Balancing Sums by
Matrix Methods}. Math. Reports. \textbf{17}(67) (2015), 225-233.

\bibitem{R-S} M. Rubajczyk and A. Szynal-Liana.\textit{Cobalancing Hybrid
Numbers}. Annales Universitatis Mariae Curie-Sklodowska, Sectio A 
Mathematica \textbf{78}(1)(2024), 87-95.

\bibitem{S-D} S.F. Santana and J.L. Diaz Barrero. \textit{Some Properties of
Sums Involving Pell Numbers}. Missouri Journal of Mathematical Science  
\textbf{18}(1)(2006), 33-40.

\bibitem{szalay} L. Szalay. \textit{On the Resolution of Simultaneous Pell
Equations}. Ann. Math. Inform. \textbf{34}(2007), 77-87.

\bibitem{tkcnx} A. Tekcan. \textit{Almost Balancing, Triangular and Square
Triangular Numbers}. Notes on Number Theory and Discrete Maths. \textbf{25}
(1)(2019), 108-121.

\bibitem{tkcnxy} A. Tekcan. \textit{Sums and Spectral Norms of all Almost
Balancing Numbers}. Creat. Math. Inform. \textbf{28}(2)(2019), 203-214.

\bibitem{alper} A. Tekcan and A. Erdem. \textit{$t$-Cobalancing Numbers and $%
t$-Cobalancers}. Notes on Number Theory and Discrete Maths.  \textbf{26}%
(1)(2020), 45-58.





\bibitem{ayd} A. Tekcan and S. Ayd\i n. \textit{On $t$-Balancers, $t$%
-Balancing Numbers and Lucas $t$-Balancing Numbers}.  Libertas Mathematica 
\textbf{41}(1)(2021), 37-51.

\bibitem{meryem} A. Tekcan and M. Y\i ld\i z. \textit{Balcobalancing Numbers
and Balcobalancers}. Creative Mathematics and Informatics \textbf{30}
(2)(2021), 203-222.

\bibitem{meryem1} A. Tekcan and M. Y\i ld\i z. \textit{Balcobalancing
Numbers and Balcobalancers II}. Creative Mathematics and Informatics \textbf{%
\ 31}(2)(2022), 247-258.

\bibitem{alper1} A. Tekcan and A. Erdem. \textit{General terms of all Almost Balancing Numbers of First and Second Type.} Communications in Mathematics \textbf{31}(1)(2023), 155-167.

\bibitem{ecem} A. Tekcan and E. Akg\"{u}\c{c}. \textit{Almost Neo
Cobalancing Numbers.} Notes on Number Theory and Discrete Mathematics 
\textbf{31}(1)(2025), 113-126.

\bibitem{meryemx1} A. Tekcan and M. Y\i ld\i z. \textit{Neo Cobalancing
Numbers}. Submitted.

\bibitem{tekneob} A. Tekcan. \textit{Almost Neo Balancing Numbers}. 
Submitted.

\bibitem{tengely} S. Tengely. \textit{Balancing Numbers which are Products
of Consecutive Integers}. Publ. Math. Deb. \textbf{83}(1-2)(2013), 197-205. 
\end{thebibliography}
\end{document}